\documentclass{article}
\usepackage[utf8]{inputenc}
\usepackage{amsmath}
\usepackage{amssymb}
\usepackage{amsthm}
\usepackage{verbatim}
\usepackage{cite}
\usepackage{graphicx}
\usepackage{subcaption}
\usepackage[shortlabels]{enumitem}
\usepackage{xcolor}
\usepackage{bm}
\usepackage{float}
\usepackage{lineno}
\usepackage{wasysym}
\usepackage{amsfonts}
\usepackage[normalem]{ulem}
\usepackage{mathtools} 

\newtheorem{theorem}{Theorem}[section]

\newtheorem{assumption}[theorem]{Assumption}

\newcommand{\AAA}{{\mathsf{A}}}
\newcommand{\AK}{\mathsf{A}_k}
\newcommand{\BB}{{\mathsf{B}}}

\newcommand{\CC}{{\mathsf{C}}}
\newcommand{\Cg}{\mathsf{C}_g}
\newcommand{\Cgs}{\mathsf{C}_{g^*}}

\newcommand{\II}{{\mathsf{I}}}
\newcommand{\PP}{{\mathsf{P}}}

\newcommand{\GG}{{\mathcal{G}}}
\newcommand{\FF}{{\mathcal{F}}}

\newcommand{\Wg}{{\mathsf{W}}_{\!g}}

\newcommand{\wgi}{{w}_{i}^{g}}

\newcommand{\xx}{\mathbf{x}}
\newcommand{\yy}{\mathbf{y}}

\newcommand{\pp}{\mathbf{p}}

\newcommand{\xg}{\mathbf{x}^g}
\newcommand{\xxg}{\mathbf{x}_g}
\newcommand{\xxgs}{\mathbf{x}_{g^{*}}}
\newcommand{\xxsgs}{\mathbf{x}^*_{g^{*}}}

\newcommand{\yyg}{\mathbf{y}_{g}}

\newcommand{\tg}{\tilde{g}}

\DeclarePairedDelimiter{\norm}{\lVert}{\rVert}

\newcommand{\ee}{\mathbf{e}}

\newcommand{\egi}{{\mathbf{e}_{i}^g}}

\newcommand{\eGi}{{\mathbf{e}_{g(i)}}}
\newcommand{\eGiOne}[1]{{\mathbf{e}_{g_1(#1)}}}
\newcommand{\eGiTwo}[1]{{\mathbf{e}_{g_2(#1)}}}

\newcommand{\bb}{{\mathbf{b}}}
\newcommand{\grouplasso}{Group Lasso }
\DeclareMathOperator*{\argmin}{arg\,min}

\title{Weighted \grouplasso for a static EEG problem}
\author{Ole L{\o}seth Elvetun\thanks{Faculty of Science and Technology, Norwegian University of Life Sciences, {\AA}s, Norway. Email: ole.elvetun@nmbu.no.} and Bj{\o}rn Fredrik Nielsen\thanks{Faculty of Science and Technology, Norwegian University of Life Sciences, {\AA}s, Norway. Email: bjorn.f.nielsen@nmbu.no.} and Niranjana Sudheer\thanks{Faculty of Science and Technology, Norwegian University of Life Sciences, {\AA}s, Norway. Email: niranjana.sudheer@nmbu.no.}}

\begin{document}

\maketitle
\begin{abstract}
 We investigate the weighted \grouplasso formulation for the static inverse electroencephalography (EEG) problem, aiming at reconstructing the unknown underlying neuronal sources from voltage measurements on the scalp. By modelling the three orthogonal dipole components at each location as a single coherent group, we demonstrate that depth bias and orientation bias can be effectively mitigated through the proposed regularization framework. On the theoretical front, we provide concise recovery guarantees for both single and multiple group sources. Our numerical experiments highlight that while theoretical bounds hold for a broad range of weight definitions, the practical reconstruction quality, for cases not covered by the theory, depends significantly on the specific weighting strategy employed. Specifically, employing a truncated Moore-Penrose pseudoinverse for the involved weighting matrix gives a small Dipole Localization Error (DLE). The proposed method offers a robust approach for inverse EEG problems, enabling improved spatial accuracy and a more physiologically realistic reconstruction of neural activity.   
\end{abstract}
\section{Introduction}

Electroencephalography (EEG) source localization aims at reconstructing the unknown underlying neuronal sources from voltage measurements on the scalp. These neuronal sources are typically modelled as current dipoles, each characterized by a specific location and an orientation. Under the quasistatic approximation of Maxwell's equations, the electric potential $\phi$ within the domain of the brain $\Omega$ satisfies Poisson's equation \cite{hamalainen1993magnetoencephalography,sarvas1987basic}, 
\begin{equation}
        \begin{split}
        \nabla \cdot \bigl(\sigma\,\nabla \phi\bigr) &= -\nabla \cdot J_{p} \quad \mbox{in } \Omega, \\
        \bigl(\sigma\nabla \phi\bigr)\cdot n &= 0 \quad \mbox{on } \partial \Omega, 
    \end{split}
    \label{eq:poi_eqn}
\end{equation} 
where $\sigma$ is the electrical conductivity and $J_{p}$ denotes the primary current density. The task of computing the electric potential $\phi$ on the scalp surface $\partial\Omega$ from a given current density $J_p$ is referred to as the direct/forward problem. 

A standard model for $J_{p}$ results from assuming a single current dipole with moment $\mathbf{q}$ located at position $r_{0}$. In distributional form, this is expressed as $$  J_{p}(r) = \mathbf{q}\delta(r -r_{0}),$$ 
where $\delta(r) $ is the Dirac delta distribution. Using this representation, the image of any single current dipole under the forward operator $$\mathcal{K}: J_p \mapsto \phi |_{\partial\Omega}$$ can be represented as a linear combination of three basis dipoles with linearly independent moment directions, corresponding to the components of $\mathbf{q}$,
\begin{equation} \label{eq:Jp_details}
    J_{p}(r) = (q_1 \ee_1+ q_2 \ee_2+ q_3 \ee_3)\delta(r -r_{0}), 
\end{equation}
where $\ee_1, \ee_2, \ee_3$ are the standard unit basis vectors and $\mathbf{q} = (q_1,q_2,q_3)^T$.  

Combining a finite element discretization of \eqref{eq:poi_eqn} with a discrete version of  restricting $\phi$ to $\partial \Omega$, yields a linear system in the form $$ \AAA\xx = \bb,$$ 
where $\AAA \in \mathbb{R}^{m \times n}$ is the lead field matrix, $\bb \in  \mathbb{R}^{m} $ contains the measured EEG data and $\xx  \in \mathbb{R}^{n} $ is the unknown source vector. In this context, $\AAA$ is thus an approximation, expressed in terms of Euclidean spaces, of $\mathcal{K}$. Each candidate source location is represented by three components of $\xx$, cf. \eqref{eq:Jp_details}, corresponding to the predefined dipole moment directions at that location.

It is well known that the forward operator $\mathcal{K}$ is not continuously invertible and that this mapping has a large null space. These properties are typically inherited by the matrix $\AAA$, and hence regularization techniques are essential to ensure stable and reasonable dipole recoveries. 

A common approach to address this ill-posedness is to incorporate prior knowledge about the source structure: Often, in the case of  EEG reconstructions, it is assumed that only a limited number of brain regions are active during cognitive tasks, which leads to the concept of spatially sparse activations \cite{gramfort2012mixed}. This requires the use of sparse recovery approaches, which have been a significant focus of research in recent years; see, e.g.,  \cite{oikonomou2020novel,mannepalli2021sparse,montoya2012structured,donoho2006compressed,gorodnitsky1997sparse}. 

However, standard inversion approaches face two main challenges: \emph{depth bias} and \emph{angle bias}. Depth bias occurs when deeper sources are misclassified as superficial sources, which is known to be a limitation of classical unweighted reconstruction schemes \cite{fuchs1999linear,grave1997linear, pascual1999review}.
Several regularization techniques have been proposed to mitigate depth bias, such as sLORETA and other minimum norm estimates \cite{gorodnitsky1995neuromagnetic,lin2006assessing,lucka2012hierarchical,xu07,pascual2002standardized}. These techniques reduce the depth bias but typically result in excessive smooth reconstructions, which is undesirable in many practical applications.  

Lasso ($\ell_1$  regularization) has also shown effectiveness in EEG reconstructions \cite{duval2017sparse,fuchs2004sparse,grasmair2011necessary,tibshirani1996regression}. Moreover, weighted $\ell^1$ models using iterative reweighting and non-uniform sparsity priors have shown an improved performance, which often produces a sparse solution under specific conditions \cite{candes05,Donoho03,Fuchs04b, candes2008enhancing, garde2016sparsity,Knudsen_2015}. In our previous investigations \cite{Elv21,Elv21c,Elv22,Elv24}, we introduced formally derived weights which involve the projection operator $\PP = \AAA^\dagger\AAA$. That is, the involved basis vectors are weighted by their projection onto the orthogonal complement of the null space of $\AAA$. We have presented several recovery results based on this method, both rigorous analyses and numerical explorations. However, only for identifying scalar sources, i.e., not for computing dipole structures. 

Angle bias arises when component-wise penalties -- often introduced to reduce smoothness or correct depth bias -- also promote directions, as a side effect, aligned with the basis dipoles. This can lead to poor or incomplete reconstruction of the source structure \cite{gramfort2012mixed, ou2009distributed}. 

To address both depth and angle bias, it seems more appropriate to enforce sparsity at the group level, where the three components of 
$\xx$ corresponding to the three predefined moment directions of a single source location are treated as one group. This falls naturally into the concept of \grouplasso \cite{yuan2006model}, which enforces sparsity by selecting all or none of the variables within a group, thus avoiding angle bias and correctly obtaining the spatial location. \grouplasso has shown significant relevance in areas such as bioinformatics, machine learning, and brain imaging \cite{meier2008group,jacob2009group,jenatton2011structured,gramfort2012mixed,lim2017sparse}.

In the group formulation, we partition the source vector $\xx$ into  \(|G|\) groups  
\begin{equation*}
    \xx =(\,\xx_{g_1}^T, \xx_{g_2}^T, \dots,  \xx_{g_{|G|}}^T\,)^T
\end{equation*}
  where each $\xx_{g_i} \in \mathbb{R}^3 $ contains the three components associated with the $i^{\text{th}}$ dipole, cf. \eqref{eq:Jp_details}. The set $G$ of groups is defined below, and $| G |$ represents the number of groups.  
  Then, employing standard \grouplasso regularization, leads to the following problem:
       \begin{equation*}
       \min_{\xx} \left\{ \frac{1}{2}\|\AAA\xx - \bb\|_{2}^2 + \alpha \sum_{g \in G} \|\xx_{g}\|_{2} \right\},   
    \end{equation*}
where $\alpha > 0$ is a regularization parameter. We will not use this approach, but instead employ a suitable weighted version of it. 

Note that if each group $g$ contains only one single variable, i.e. $\xx_{g} \in \mathbb{R}$, then the $\ell_2$-norm becomes the absolute value, and hence the \grouplasso approach reduces to standard Lasso. Whereas when there is only one group that contains all the variables, i.e. $|G| = 1$, then the penalty term/regularization term simply becomes $\alpha \|\xx\|_{2}$. 

\section{Theoretical background}

\subsection{The \grouplasso formulation}
Since the introduction of the \grouplasso by Yuan and Lin \cite{yuan2006model}, the method has become a significant subject of research in statistics, optimization, and inverse problems. The framework was later refined by the standardized \grouplasso formulation \cite{simon2012standardization}, which obviates the need for orthonormalization within each group.

Several theoretical findings have been established in recent years. For instance, Bach \cite{bach2008consistency} studied \grouplasso for both finite- and infinite-dimensional groups, providing necessary and sufficient conditions for model consistency. This work also established a link between \grouplasso and multiple kernel learning \cite{rakotomamonjy2008simplemkl}. Further recovery guarantees have been analyzed in \cite{bach2008consistency,article,eldar09}, with recent work offering improved theoretical bounds for support recovery \cite{elyaderani2017improved}, notably applied in the context of structural health monitoring.

Beyond consistency, researchers have investigated the specific sparsity properties of \grouplasso, particularly for feature selection \cite{huang2009learning,cai2022sparse}. The concept of strong group sparsity was developed in \cite{huang2010benefit}, demonstrating scenarios where \grouplasso outperforms standard Lasso. These findings have motivated several extensions of the original formulation.

Significant effort has been dedicated to developing efficient numerical algorithms for solving \grouplasso problems, including coordinate descent, proximal gradient methods, and ADMM \cite{friedman2010regularization,klosa2020seagull,deng2013group}. A prominent approach is the block coordinate descent algorithm proposed for logistic regression models \cite{meier2008group}. Other works focus on optimization techniques for structured sparsity, such as overlapping and hierarchical group structures \cite{jenatton2011proximal,qi2024non,jacob2009group}.

\subsection{Application to EEG Source Localization}

There have been several studies aligned with the focus of the current paper, specifically regarding the application of \grouplasso methods to the static EEG inverse problem. As we mentioned in the introduction, the primary challenge in this context is the modelling of the free orientation of neural sources without introducing bias. To address this, works such as \cite{haufe2008sparse, gramfort2012mixed, lim2017sparse} have utilized the \grouplasso framework to bundle the three orthogonal dipole components at each spatial location. By applying an isotropic group penalty to these co-located components, these methods enforce rotational invariance, ensuring that the recovered source amplitudes are independent of the specific choice of coordinate system. This approach effectively mitigates the "grid effect" often observed in standard $\ell_1$- regularization, where reconstructed dipoles tend to align artificially with the Cartesian axes.

We contribute to this development by deriving concise recovery guarantees -- with transparent proofs -- for both Weighted Group Pursuit and Weighted \grouplasso problems. Furthermore, we demonstrate a critical distinction between theory and practice: While theoretical recovery guarantees hold for a broad range of weight choices, the empirical reconstruction quality, for cases not covered by the theory, depends significantly on the specific weighting strategy employed.

\section{Analysis}
\label{Analysis}
We formulate the weighted \grouplasso problem as follows 
\begin{equation}
    \min_{\xx}\left\{\frac{1}{2}\|\CC\xx-\BB\bb\|^2 + \alpha \sum_{g \in G} \|\Cg\xxg\| \right\}, \label{eq:varCform}
\end{equation}
where 
\begin{itemize}
    \item $\AAA \in \mathbb{R}^{m \times n}$, $n > m$, i.e., $\AAA$ has a non-trivial null space.
    \item $\BB \in \mathbb{R}^{q \times m}$ and $\CC := \BB\AAA$.
    \item $G$ is the collection of groups, 
    \begin{equation*}
        G = \left\{ g_1,g_2, \ldots, g_{|G|} \right\}
    \end{equation*}
    and $g_1,g_2, \ldots, g_{|G|}$ are Euclidean vectors containing the indexes associated with each group. For example, 
    \begin{equation*}
        g_1=(g_1(1), g_1(2), \ldots, g_1(|g_1|))^T, 
    \end{equation*}
    where $g_1(1), g_1(2), \ldots, g_1(|g_1|)$ are the indexes of the members of group $g_1$. $|g_1|$ denotes the cardinality of group $1$.
    \item $\xxg$ is the sub-vector of $\xx$ and $\Cg$ is the sub-columns of $\CC$ associated with group $g \in G$.
\end{itemize}
Note that \eqref{eq:varCform} is not expressed in terms of the original lead field matrix $\AAA$, but instead we involve an additional matrix $\BB$ to obtain a more flexible formulation. 
Setting $\BB = \mathsf{I}$, the identity matrix, yields the standardized \grouplasso formulation first introduced in \cite{simon2012standardization}. Multiplying $\AAA \xx = \bb$ with $\BB$ leads to the equation $\BB \AAA \xx = \BB \bb$. Hence, $\BB$ can be "baked" directly into the forward problem, indicating that the role of $\BB$ might be rather irrelevant, but we will see in the numerical section that the choice of $\BB$ can, in practice, strongly influence the recovery properties of the method. We have not been able to prove theoretically why this is the case, but for the standard Lasso formulation, theoretical findings concerning this issue can be found in \cite{ENS25}.

We will now prove exact recovery results for both single and multiple group scenarios. These proofs are fairly short and rather transparent.

\subsection{Group pursuit}

Our first result is an exact recovery result for a single group source. In order to obtain uniqueness, we need to assume that any data explainable by a single group cannot be explained by any other single group.
\begin{assumption} \label{assumption_uniqueness}
    For any $g_1, g_2 \in G$, $g_1 \neq g_2$, we assume that the vectors 
    \begin{equation*}
        \CC \eGiOne{1}, \CC \eGiOne{2}, \ldots, \CC \eGiOne{|g_1|},  \CC \eGiTwo{1}, \CC \eGiTwo{2}, \ldots, \CC \eGiTwo{|g_2|}
    \end{equation*}
    are linearly independent. 
\end{assumption}
This assumption assures that there can {\em not} exist scalars $\{ a_i \}$ and $\{ b_i\}$, not all of them zero, and a constant $c$ such that 
\begin{equation*}
    \sum_i a_i \CC \eGiOne{i} = c \sum_i b_i \CC \eGiTwo{i}
\end{equation*}
\begin{theorem}[Single group source]\label{thm:grouppursuit}
    Assume that $\xx^*$ is a single group source, i.e., $\textnormal{supp}(\xx^*) = g^*$ for some $g^* \in G$.  
    Then 
    \begin{equation*}
         \xx^* \in \argmin_{\xx} \sum_{g \in G} \|\Cg\xxg\| \quad \textnormal{subject to} \quad \CC\xx = \CC\xx^*.
    \end{equation*}
    If, in addition, Assumption \ref{assumption_uniqueness} is satisfied, then $\xx^*$ is the unique solution of this single group pursuit problem. 
\end{theorem}
We note that this result follows from the block (group) restricted isometry property introduced in \cite{eldar09}. We nevertheless, for the sake of completeness, include our proof of this theorem as it is short, self-contained and serves as a simplified counterpart to the argument used in the multiple-source setting below.

\begin{proof}
    Let $\yy \neq \xx^*$ be any solution of $\CC\yy = \CC\xx^*$. Since $\xx^*$ has support on a single group, i.e., $\textnormal{supp}(\xx^*) = g^*$, it follows that
    \begin{eqnarray*}
        \sum_{g\in G} \|\Cg\xxg^*\| &=& \|\Cgs\xxsgs\| \\
        &=& \|\CC\xx^*\| \\ &=& \|\CC\yy\| \\ &=& \norm[\Big]{\sum_{g\in G} \Cg\yyg} \\ &\leq& \sum_{g\in G}\|\Cg\yyg\|.
    \end{eqnarray*}
    Note that if Assumption \ref{assumption_uniqueness} is satisfied, then the triangle inequality in this argument becomes strict.
\end{proof}

The next theorem concerns the identifiability of multiple groups. Under the assumption stated below, we present a result that, to the best of our knowledge, is new.
We will assume that the group supports of the images under $\CC^T \CC$ of the involved group members are disjoint. Here, the group image of $g \in G$ is defined as follows  
\begin{equation*}
    \zeta_g=\left\{ \tilde{g} \in G \, | \, \exists \ \xxg \mbox{ such that } \tilde{g}(j) \in \mbox{supp}(\CC^T \Cg \xxg) \mbox{ for some } j \in \{1,2, \ldots, |\tilde{g}| \} \right\}. 
\end{equation*}

\begin{assumption} \label{assumption:disjoint_group_images}
    For $J \subset G$ we assume that 
    \begin{equation*} 
        \zeta_{g_1} \bigcap \zeta_{g_2} = \emptyset \quad \forall g_1, g_2 \in J, g_1 \neq g_2.
    \end{equation*}
\end{assumption}
If $J$ is such that this assumption holds, then 
\begin{equation} \label{assumption:disjoint_CTC}
          \mbox{supp}(\CC^T \CC_{g_1}\xx_{g_1}) \bigcap  \mbox{supp}(\CC^T \CC_{g_2}\xx_{g_2})  = \emptyset \quad \forall g_1, g_2 \in J, g_1 \neq g_2.
    \end{equation}
    This is verified as follows: Assume that $q \in \mbox{supp}(\CC^T \CC_{g_1}\xx_{g_1})$ and that $q \in \mbox{supp}(\CC^T \CC_{g_2}\xx_{g_2})$. Then there exists $\tilde{g}_1 \in \zeta_{g_1}$ and $j_1$ such that $\tilde{g}_1(j_1)=q$. But, since $q \in \mbox{supp}(\CC^T \CC_{g_2}\xx_{g_2})$ it follows from the definition of $\zeta_{g_2}$ that $\tilde{g}_1$ also belongs to $\zeta_{g_2}$, i.e, $\zeta_{g_1} \bigcap \zeta_{g_2} \neq \emptyset$. 

\begin{theorem}[Recovery of disjoint groups] \label{thm:several_groups}
    Let $J \subset G$ be a set of group indices satisfying Assumption \ref{assumption:disjoint_group_images} and let $\xx^*$ be a vector which has support on the groups in $J$. 
    
    Then 
    \begin{equation*}
         \xx^* = \argmin_{\xx} \sum_{g \in G} \|\CC_g\xxg\| \quad \textnormal{subject to} \quad {\CC}\xx = {\CC}\xx^*.
    \end{equation*}
\end{theorem}

\begin{proof}
    Let $\yy$ be any vector satisfying $\CC\yy = \CC\xx^*$. Then
    \begin{eqnarray*}
        \sum_{g \in G} \|\Cg\xxg^*\| &=& \sum_{g \in J} \|\Cg\xxg^*\|  
        = \sum_{g \in J} \sqrt{(\Cg\xxg^*, \Cg\xxg^*)} \\
        &=& \sum_{g \in J} \sqrt{(\Cg\xxg^*,\CC\xx^*)}  
        = \sum_{g \in J} \sqrt{(\Cg\xxg^*, \CC\yy)} \\
        &=& \sum_{g \in J} \sqrt{\left(\Cg\xxg^*, \sum_{\tg \in \zeta_g}\CC_{\tg}\yy_{\tg}\right)} \\
        &\leq& \sum_{g \in J} \sqrt{\norm{\Cg\xxg^*}}\sqrt{\norm[\Big]{\sum_{\tg \in \zeta_g}\CC_{\tg}\yy_{\tg}}} \\ 
        &\leq& \sqrt{\sum_{g \in J} \|\Cg\xxg^*\|}\sqrt{\sum_{g \in J} \norm[\Big]{\sum_{\tg \in \zeta_g}\CC_{\tg}\yy_{\tg}}} \\
        &\leq& \sqrt{\sum_{g \in G} \|\Cg\xxg^*\|}\sqrt{\sum_{g \in J} \norm[\Big]{\sum_{\tg \in \zeta_g}\CC_{\tg}\yy_{\tg}}}.
    \end{eqnarray*}
    Consequently, also employing the triangle inequality, 
    \begin{align*}
        \sum_{g \in G} \|\Cg\xxg^*\| &\leq \sum_{g \in J} \norm[\Big]{\sum_{\tg \in \zeta_g}\CC_{\tg}\yy_{\tg}} \\
        & \leq \sum_{g \in J} \sum_{\tg \in \zeta_g} \norm[\Big]{\CC_{\tg}\yy_{\tg}} \\ &\leq \sum_{g \in G}\|\Cg\yyg\|,
    \end{align*}
    where the last inequality follows from the assumption that the sets $\{ \zeta_g \}_{g \in J}$ are mutually disjoint and the fact that $$\bigcup_{g \in J} \zeta_g \subset G.$$
\end{proof}

\subsection{\grouplasso}
In practice, the group pursuit problem is rarely solved: Instead, one typically considers its variational counterpart given in \eqref{eq:varCform}. Due to the trade-off between the fidelity and regularization terms, exact recovery is generally not achievable. However, as the next theorem demonstrates, the solution obtained in this regularized setting equals a rescaled version of the true group source.

\begin{theorem}\label{thm:grouplasso}
  Assume that $\xx^*$ is a single group source, i.e., $\textnormal{supp}(\xx^*) = g^*$ for some $g^* \in G$. Then
  \begin{equation*}
      \xx_\alpha = \gamma_\alpha\xx^*  \in \argmin_\xx \left\{ \frac{1}{2}\|\CC\xx - \CC\xx^*\|^2 + \alpha \sum_{g \in G} \|\Cg\xxg\|\right\},
  \end{equation*}
  where $\gamma_\alpha = 1 - \frac{\alpha}{\|\Cgs\xxgs^*\|}$ and provided that $0 < \alpha < \|\Cgs\xxgs^*\|$.
\end{theorem}

\begin{proof}
    Our proof consists of two steps. First, we determine a candidate minimizer by considering possible solutions in the form $\xx = \gamma\xx^*$, $\gamma \in \mathbb{R}$. 
    Second, we use elementary inequalities to demonstrate that this solution is also a global minimum.

    Focusing on the first step, we insert $\xx = \gamma\xx^*$ in the cost-functional. Recalling that $\CC\xx^* = \Cgs\xxgs^*$, we get
        \begin{equation*}
        \frac{1}{2}\|\CC\xx-\CC\xx^*\|^2 + \alpha\sum_{g \in G} \|\CC\xx\| = \frac{1}{2} (1-\gamma)^2\|\Cgs\xxgs^*\|^2 + \alpha \gamma \|\Cgs\xxgs^*\|.
    \end{equation*}
    Minimizing this expression with respect to $\gamma$ is a straightforward task and yields that 
    \begin{equation*}
        \gamma_\alpha = 1 - \frac{\alpha}{\|\Cgs\xxgs\|}.
    \end{equation*}
    Note that $\gamma_\alpha > 0$ when $0 < \alpha < \|\Cgs\xxgs^*\|$. 
        
    For the second part of the proof, we employ the Cauchy-Schwarz inequality to get 
    \begin{align}
        \nonumber
        \frac{1}{2}\|\CC\xx-\CC\xx^*\|^2 &= \frac{1}{2}\|\CC\xx\|^2 -(\CC\xx,\CC\xx^*) + \frac{1}{2}\|\CC\xx^*\|^2 \\ 
        \label{eq:group_lasso_first_inequality}
        &\geq
        \frac{1}{2}\|\CC\xx\|^2 - \|\CC\xx\|\|\CC\xx^*\| + \frac{1}{2}\|\CC\xx^*\|^2, 
    \end{align}
    and the triangle inequality yields
    \begin{equation} \label{eq:group_lasso_second_inequality}
        \sum_{g \in G} \|\Cg\xxg\| \geq \norm[\Big]{\sum_{g \in G}\Cg\xxg} = \|\CC\xx\|.
    \end{equation}
    By combining these expression we derive that
    \begin{align*}
        \overbrace{\frac{1}{2}\|\CC\xx-\CC\xx^*\|^2 + \alpha \sum_{g \in G} \|\Cg\xxg\|}^{:=\GG(\xx)} &\geq \overbrace{\frac{1}{2}\|\CC\xx\|^2 - \|\CC\xx\|\|\CC\xx^*\| + \frac{1}{2}\|\CC\xx^*\|^2 + \alpha \|\CC\xx\|}^{:= \FF(\xx)} \\
        & = \frac{1}{2} \rho^2 - \rho \|\CC\xx^*\| + \frac{1}{2}\|\CC\xx^*\|^2 + \alpha \rho,
    \end{align*}
    where $\rho =\|\CC\xx||$. 
    
    Minimizing the last expression with respect to $\rho$ yields that 
    \begin{equation*}
        \rho^* = \|\CC\xx^*\| - \alpha.
    \end{equation*}
    Since $\|\CC \gamma_\alpha\xx^*\| = \|\CC\xx^*\| - \alpha = \rho^*$, $\gamma_\alpha\xx^*$ is a minimizer for $\FF$. With $\xx=\gamma_\alpha\xx^*$, both \eqref{eq:group_lasso_first_inequality} and \eqref{eq:group_lasso_second_inequality} hold with equality and therefore $\GG(\gamma_\alpha\xx^*) = \FF(\gamma_\alpha\xx^*)$. We thus conclude that $\gamma_\alpha\xx^*$ also is a minimizer of $\GG$. 
\end{proof}

We have not succeeded in proving a regularized version of Theorem \ref{thm:several_groups}. In the several groups case, the first step of the proof of Theorem \ref{thm:grouplasso} can be generalized in a straightforward manner. However, the second part becomes very challenging. Also, analyzing the issue in terms of the standard first-order optimality condition leads to very involved investigations of subgradients and will doubtfully yield a very transparent argument. We hence regard the matter as an open problem.  

\section{Numerical experiments}
In this study, we explore the quality of the weighted \grouplasso formulation for EEG source reconstruction. We use the open-source MATLAB based toolbox SEREEGA \cite{krol2018sereega} to generate event related EEG data. For all the experiments, a pre-generated New York Head (ICBM-NY) lead field is used as the forward model. This provides the lead field matrix $\AAA \in \mathbb{R}^{m \times 3n}$, where $m$ is the number of EEG electrodes and $n = 74382$ is the number of dipole source positions distributed across the brain. Each dipole is represented by three orientation components $(x,y,z)$, resulting in $3n$ columns. The cortical volume of the ICBM-NY model has a size of $144\times179\times133$ $mm^3$, which roughly corresponds to the dimensions of a typical adult human cerebral cortex.

For the inverse problem, we construct a reduced lead field matrix $\AAA_r \in \mathbb{R}^{m \times 3p}$ by selecting a random subset of $p = 20000$ dipole positions from the lead field $\AAA$ , with a minimum separation of $2$ mm between them. The remaining dipoles are then used to generate the forward data, hence avoiding an inverse crime. Note that with this setup, it is impossible to perfectly recover the positions of the true dipole sources

The weighted \grouplasso formulation, 
\begin{equation*}
    \min_{\xx}\left\{\frac{1}{2}\|\CC\xx-\BB\bb\|^2 + \alpha \sum_{g \in G} \|\Cg\xxg\| \right\},
\end{equation*}
where $\CC = \BB\AAA$, is solved using the Block Coordinate Descent algorithm \cite{qin2013efficient}. We compute simulations for the following cases: 
\begin{itemize}
    \item $\BB = \II$, the identity matrix, which implies that we do not modify the forward matrix.
    \item $\BB = \AK^\dagger$, the truncated SVD of $\AAA$, using $k = 150$ singular vectors and their associated singular values. This choice is motivated by our attempt to approximately satisfy \eqref{assumption:disjoint_CTC}, which we elaborate on for the standard sparsity case, i.e., groups with only one member, in our previous work \cite{ENS25}.
    
\end{itemize}

A Gaussian noise of 1\% is added to the generated data. To avoid the simplest case where the dipoles are very close to the EEG electrodes, we only use true sources located a minimum of $15$ mm away from the closest EEG electrode. For each simulation, an adequate size of the regularization parameter $\alpha$ is obtained by applying Morozov's discrepancy principle.
To measure the accuracy of the localization, we calculate the dipole localization error (DLE): The Euclidean distance between the position of the true source and the position of the estimated dipole. We also compute the dipole orientation error (DOE): The angle between the true and the inversely estimated dipole orientations, measured in radians.

We perform 100 simulations with a single source activation for each trial, and we use a 228-channel EEG recording. The mean DLE values obtained were $5.8461$ mm and $12.6366$ mm when $\BB = \AK^\dagger$ and $\BB = \II$, respectively, with a theoretical minimum of $1.6485$ mm.   The corresponding mean DOE values were $0.2207$ rad $(12.64^\circ)$ and $0.2448$ rad $(14.02^\circ)$ for $\BB = \AK^\dagger$ and $\BB = \II$, respectively.

A scatter plot illustrating the depth bias observed in these simulations is shown in Figure \ref{plot:depth_bias}.
The x-axis represents the depth of the true source, and the y-axis represents the depth of the inverse source, with the ideal case represented as the red dotted line. 
Note that the performance of the method is acceptable if the results are closely clustered around the line $y = x$. We can observe that in the case of $\BB  = \II$, some values are spread out and lie below the dotted line, indicating less precision. However, it is clearly visible that the accuracy of the source localization improves significantly when employing $\BB=\AK^\dagger$. 

Figures \ref{plot:comparision_two} and \ref{plot:comparision_three} display comparisons of the true and the inverse source locations for cases where two and three sources are activated simultaneously. These plots are generated using a SEREEGA toolbox function, which displays a 2D representation of the brain in coronal, sagittal, and axial planes, respectively.

    \begin{figure}[H]
\centering
\begin{subfigure}{0.65\textwidth}
    \centering
    \includegraphics[width=\linewidth]{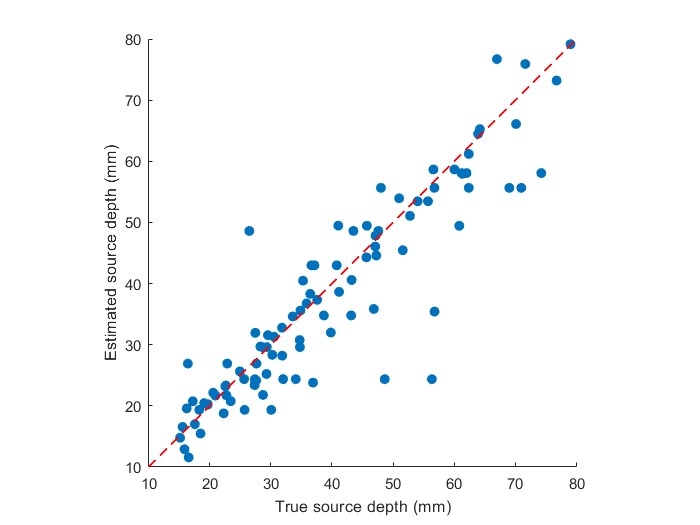}
    \caption{$\BB = \II $}
\end{subfigure}\par
\begin{subfigure}{0.65\textwidth}
    \centering
    \includegraphics[width=\linewidth]{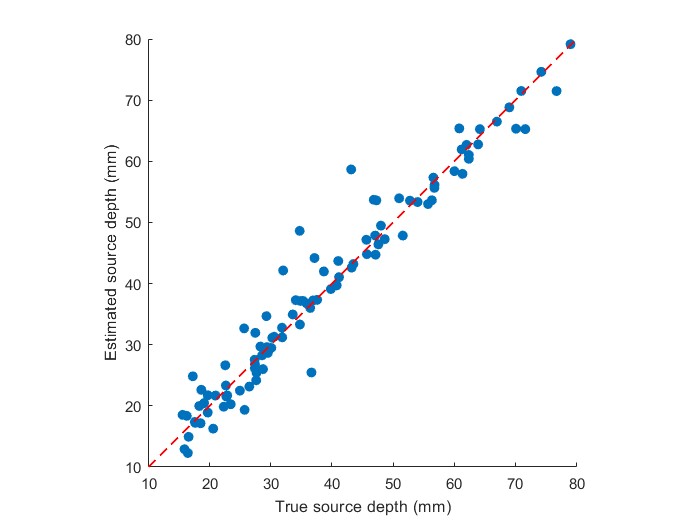}
    \caption{ $\BB = \AK^\dagger $}
\end{subfigure}
\caption{Scatter plots for the true source depth versus the estimated source depth. If all points are located on the dashed red line $x = y$, there is no depth bias. Simulations for two different choices of $\BB$.}
\label{plot:depth_bias}
\end{figure}

\begin{figure}[H]
\centering
\begin{subfigure}{0.75\textwidth}
    \centering
    \includegraphics[width=\linewidth]{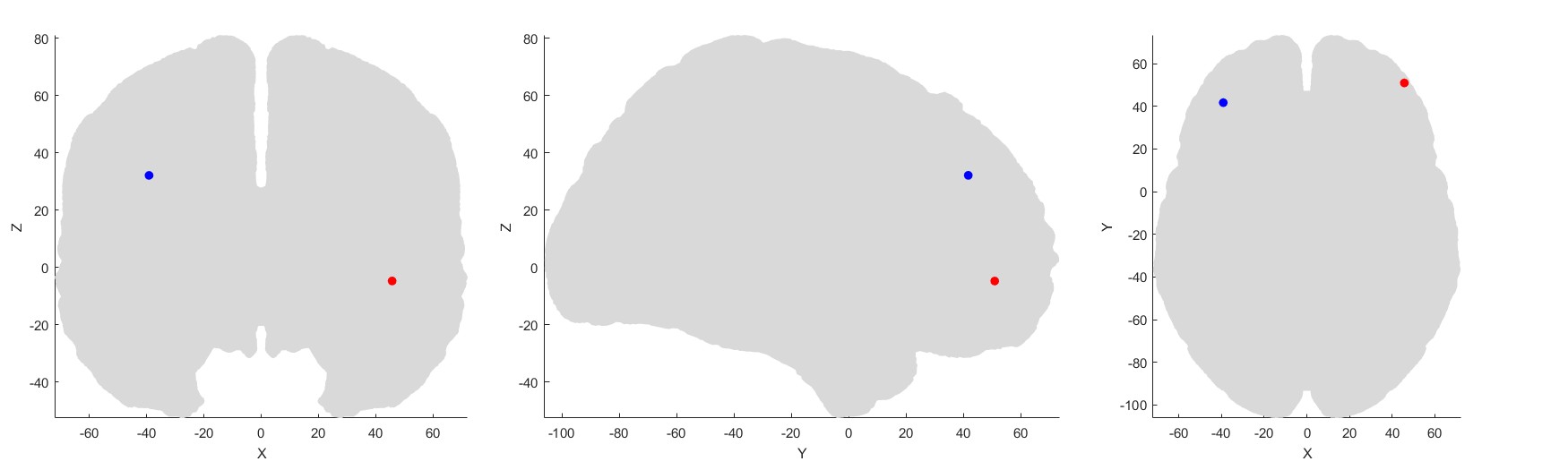}
    \caption{True sources}
\end{subfigure}\par
\begin{subfigure}{0.75\textwidth}
    \centering
    \includegraphics[width=\linewidth]{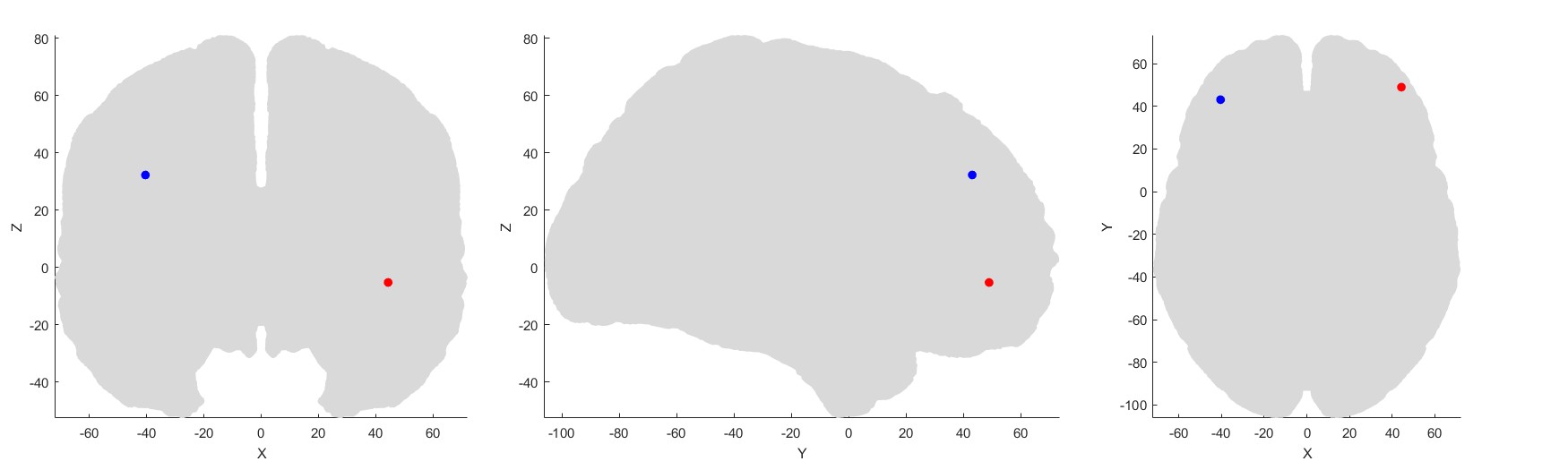}
    \caption{Estimated (inverse) sources when $\BB = \AK^\dagger$}
\end{subfigure}
\begin{subfigure}{0.75\textwidth}
    \centering
    \includegraphics[width=\linewidth]{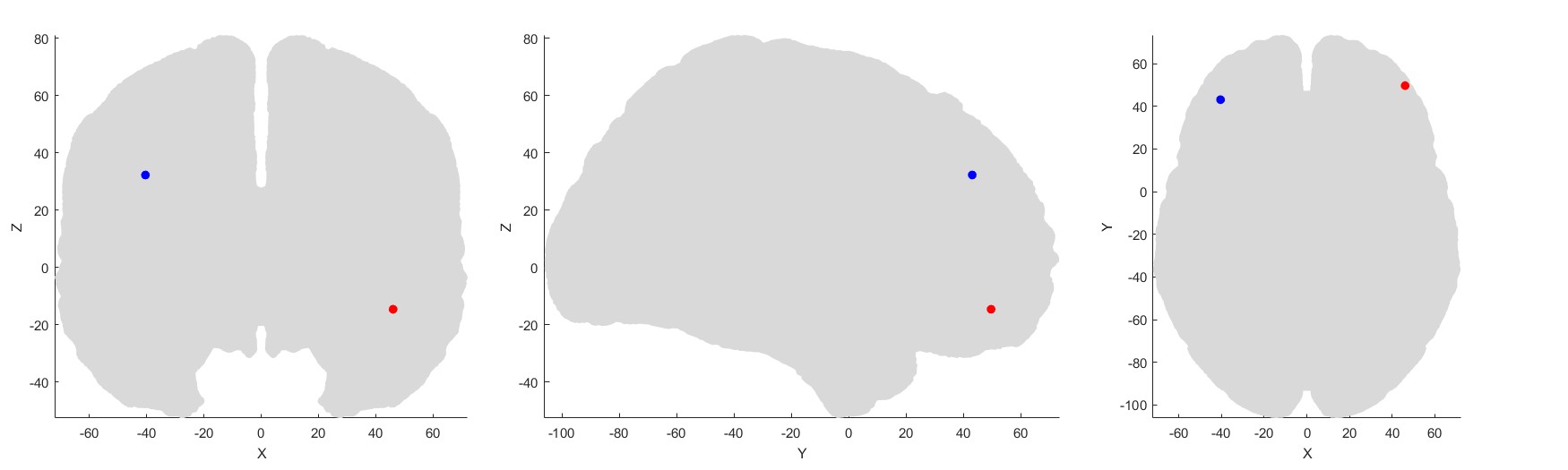}
    \caption{Estimated (inverse) sources when $\BB = \II$}
\end{subfigure}
\caption{Two-source reconstruction with average DOE values of $0.0455$ rad $(2.60^\circ)$ and $0.1378$ rad $(7.89^\circ)$ when $\BB = \AK^\dagger$ and $\BB = \II$, respectively.}
\label{plot:comparision_two}
\end{figure}

\begin{figure}[H]
\centering
\begin{subfigure}{0.75\textwidth}
    \centering
    \includegraphics[width=\linewidth]{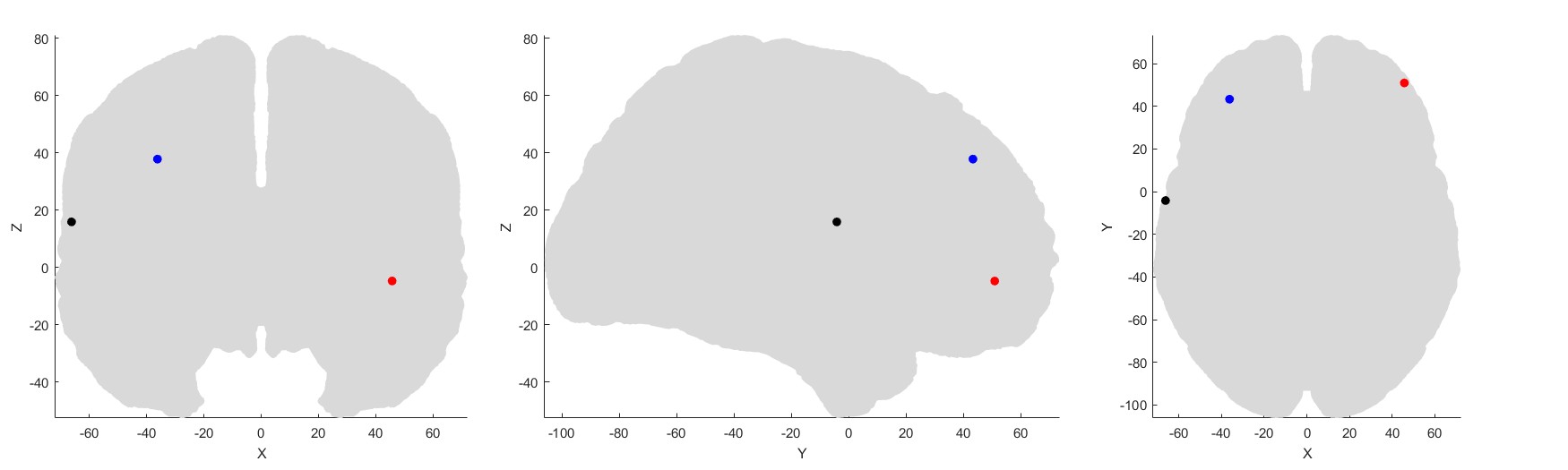}
    \caption{True sources}
\end{subfigure}\par
\begin{subfigure}{0.75\textwidth}
    \centering
    \includegraphics[width=\linewidth]{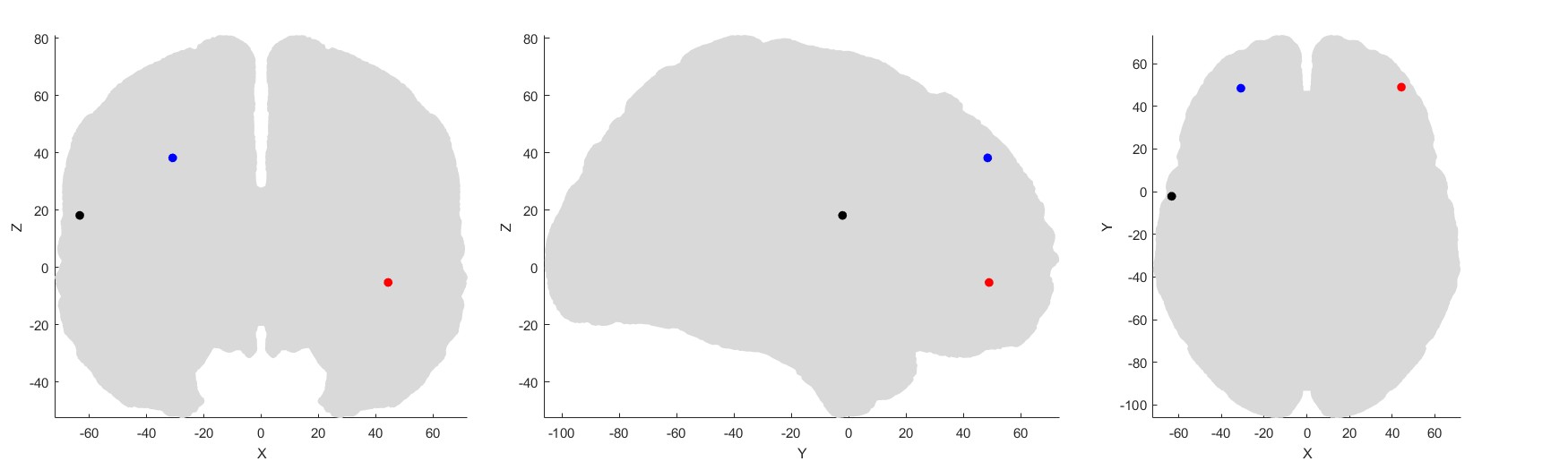}
    \caption{Estimated (inverse) sources when $\BB = \AK^\dagger$}
\end{subfigure}
\begin{subfigure}{0.75\textwidth}
    \centering
    \includegraphics[width=\linewidth]{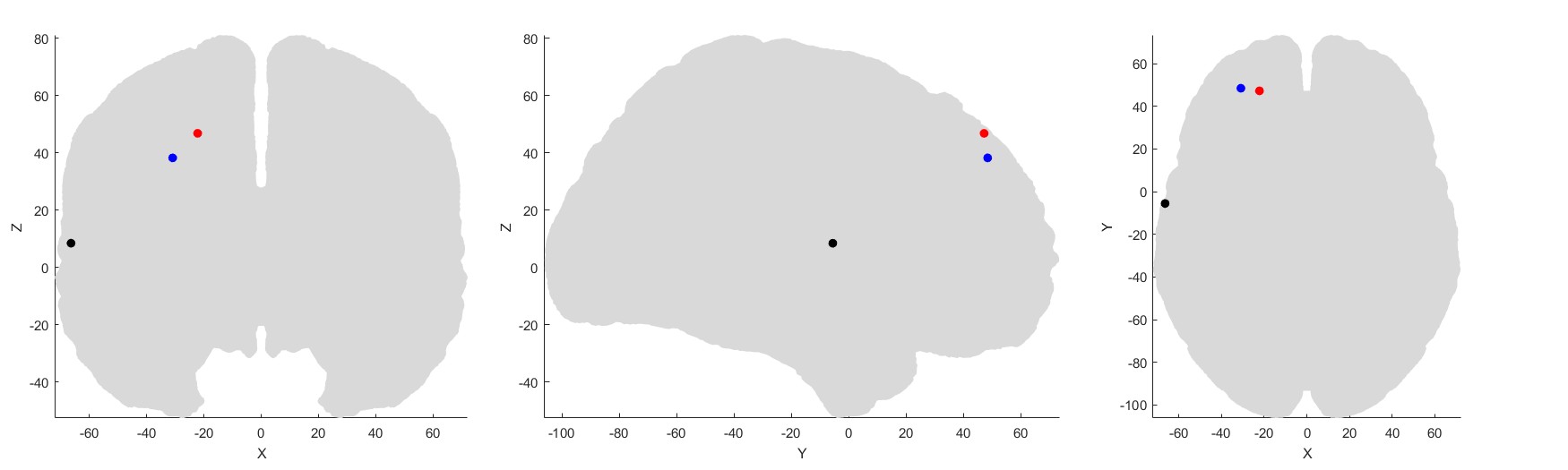}
    \caption{Estimated (inverse) sources when $\BB = \II$}
\end{subfigure}
\caption{Three-source reconstruction with average DOE values of $0.2985$ rad $(17.1^\circ)$ and $0.5931$ rad $(33.9^\circ)$ when $\BB = \AK^\dagger$ and $\BB = \II$, respectively.}
\label{plot:comparision_three}
\end{figure}

\section{Conclusion}
This work has introduced a weighted Group Lasso framework tailored for the static EEG inverse problem, emphasizing the joint recovery of source location and orientation. By treating the three dipole components at each spatial position as a unified group, the proposed method effectively addresses depth and orientation biases that commonly hinder traditional sparse reconstruction techniques. Our theoretical analysis provides transparent recovery guarantees for both single and multiple group scenarios, highlighting the interplay between weighting strategies and identifiability conditions.

Empirical results further demonstrate that while theoretical consistency holds under broad assumptions, practical performance is strongly influenced by the choice of weighting operator. In particular, incorporating a truncated Moore-–Penrose pseudoinverse significantly improves localization accuracy and orientation fidelity compared to standard weighted formulations. These findings suggest that weighted group-based regularization offers a promising direction for EEG source imaging, bridging the gap between theoretical rigor and physiological plausibility. Future research should explore the possibility of extending the framework to dynamic or multi-modal neuroimaging settings, potentially enabling more robust and interpretable reconstructions in real-world applications.

\bibliographystyle{abbrv} 
\bibliography{reference}
\end{document}